\documentclass[A4,11pt]{article}

\title{A combinatorial perspective on the Kemeny constant and more}
\author{Luis Fredes\footnote{Univ. Bordeaux, CNRS, Bordeaux INP, IMB, UMR 5251, F-33400 Talence, France  } \textsf{ and } Jean Francois Marckert  \footnote{Univ. Bordeaux, CNRS, Bordeaux INP, LaBRI, UMR 5800, F-33400 Talence, France}}
\date{September 2023}
\usepackage{xspace,fourier,graphicx}

\usepackage[utf8]{inputenc}
\usepackage{amsthm}
\usepackage{amsmath}
\usepackage{amssymb,amsfonts, ulem}
\usepackage{hyperref,stmaryrd}
\usepackage{cleveref}
\usepackage{xcolor,todonotes}
\usepackage[left=2.2cm,right=2.2cm,top=2cm,foot=2cm]{geometry}

\usepackage[english]{babel}

\newcommand{\Tr}{{\sf Tr}}
\newcommand{\Adj}{{\sf Adj}}
\newcommand{\N}{\mathbb{N}}
\newcommand{\R}{\mathbb{R}}
\newcommand{\C}{\mathbb{C}}

 \def \Id{{\sf Id}}
\def \1{\textbf{1}}

\def \E{\mathbb{E}}

\def \Geo{{\sf Geo}}
\def \G{{\sf G}}

\def \Id{{\sf Id}}

\def \M{{\sf M}}

\def \N{\mathbb{N}}

\def \P{{\mathbb{P}}}

\def \R{\mathbb{R}}

\def \bK{{\bf K}}

\def \bQ{{\bf Q}}

\def \bZ{{\bf Z}}
\def \bar{\overline}
\def \ba{\begin{align}}
\def \ba{{\bf a}}

\def \ben{\begin{eqnarray}}
\def \beq{\begin{equation}}
\def \be{\begin{eqnarray*}}

\def \bma{\begin{pmatrix}}
\def \bpar#1{\left\{\begin{array}{#1} }

\def \build#1#2#3{\mathrel{\mathop{\kern 0pt#1}\limits_{#2}^{#3}}}

\def \captionn#1{\begin{center}\begin{minipage}{17cm}\sf\caption{\small \textsf{#1}}\end{minipage}\end{center}}

\def \dis{\displaystyle}
\def \ea{\end{align}}

\def \een{\end{eqnarray}}
\def \ee{\end{eqnarray*}}
\def \ema{\end{pmatrix}}

\def \epar { \end{array}\right.}

\def \eq{\end{equation}}
\def \eref#1{(\ref{#1})}

\def \l{\left}

\def \r{\right}

\def \sous#1#2{\mathrel{\mathop{\kern 0pt#1}\limits_{#2}}}
\def \sur#1#2{\mathrel{\mathop{\kern 0pt#1}\limits^{#2}}}

\newcommand{\overbar}[1]{\overline{#1}}

\def \E{\mathbb{E}}
\theoremstyle{plain}

\newtheorem{lem}{Lemma}

\newtheorem{pro}[lem]{Proposition}
\newtheorem{theo}[lem]{Theorem}

\newtheorem{rem}[lem]{Remark}

\newtheorem{cor}[lem]{Corollary}

\numberwithin{equation}{section}
\numberwithin{lem}{section}
\numberwithin{exa}{section}

\def \1{\textbf{1}}

\newcommand{\compact}{ \topsep0pt   \itemsep=0pt   \partopsep=0pt   \parsep=0pt}

\newcounter{c}
\def \bir{\begin{itemize}\compact \setcounter{c}{0}}
\def \itr{\addtocounter{c}{1}\item[($\roman{c}$)]} 
\def \eir{\end{itemize}\vspace{-2em}~}

\newcounter{d}
\def \bia{\begin{itemize}\compact \setcounter{d}{0}}
\def \eia{\end{itemize}\vspace{-2em}~}

\newcounter{b}
\def \bi{\begin{itemize}\compact \setcounter{b}{0}}

\def \ei{\end{itemize}\vspace{-2em}~}

\theoremstyle{plain}
\numberwithin{equation}{section}
\numberwithin{lem}{section}

\begin{document}
 
\maketitle 
\begin{abstract} 
 Let $\M$ be an irreducible transition matrix on a finite state space $V$. For a Markov chain $C=(C_k,k\geq 0)$ with transition matrix $\M$, let $\tau^{\geq 1}_u$ denote the first positive hitting time of $u$ by $C$, and $\rho$ the  unique invariant measure of $\M$. Kemeny proved that if $X$ is sampled according to $\rho$ independently of $C$, the expected value of the first positive hitting time of $X$ by $C$ does not depend on the starting state of the chain: all the values $\l(\E(\tau^{\geq 1}_X~|~C_0=u), u \in V\r)$ are equal. \par
  In this paper, we show that this property follows from a more general result: the generating function \newline$\sum_{v\in V} \E\l(x^{\tau_v^{\geq 1}}~|~C_0=u\r)\det\l(\Id-x\M^{(v)}\r)$ is independent of the starting state $u$, where $\M^{(v)}$ is obtained from $\M$ by deleting the row and column corresponding to the state $v$. The factors appearing in this generating function are: first, the probability generating function of $\tau^{\geq 1}_v$, and second, the sequence of determinants $\l(\det(\Id-x\M^{(v)}),v\in V\r),$ which, for $x=1$, is known to be proportional to the invariant measure $(\rho_u,u\in V)$. 
    From this property, we deduce several further results, including relations involving higher moments of $\tau_X^{\geq 1}$, which are of independent interest. 
\end{abstract}
\section{Introduction}

\paragraph{Notation/convention.}  The identity matrix of size $d\times d$ is denoted by $\Id_d$. 

\subsection{Kemeny constant and a new functional identity}

All along the paper, $C:=(C_j,j\geq 0)$ stands for a Markov chain with irreducible transition matrix $\M$ defined on a finite set $V$ of cardinality $d$.
The invariant measures of $\M$  are all proportional  to $(\pi_v,v\in V$) defined by
 \ben \pi_v := \det\big(\Id_{d-1}-\M^{(v)}\big),~~\textrm{ for all } v \in V,\een
 where $\M^{(v)}$ is the matrix obtained from $\M$ by the suppression of the row and column associated with $v$ (this is a classical result that can be proved directly; see, e.g. \cite[Sec.1.1]{FMAB}, or by using the Markov chain tree theorem used to prove the correctness of Aldous-Broder algorithm \cite{Bro89,Al90,HLT21,FMAB}).
The invariant probability distribution of $\M$ is therefore $(\rho_u,u \in V)$ where
 \[\rho_u= \pi_u / {\bf Z} \textrm{~~~for~~~}{\bf Z}=\sum_{v\in V} \pi_v.\]
For any integer $t\geq 0$,  any state $u\in V$, let \ben\tau_u^{\geq t}=\inf\big\{k\geq t: C_k=u\big\}\een the first hitting time of state $u$ after (and including) time $t$.
 For any real-valued function $f$ of $C$, we write $\E_u(f(C))$ for the expectation of $f(C)$ conditional on $C_0 = u$, and $\P_u$ for the distribution of $C$ conditional on $C_0=u$.
 
For $u\in V$, set 
 \[Q^{\geq 1}(u)  :=\sum_{v \in V}  \E_u\l(\tau_v^{\geq 1}\r)\pi_v.\]
 \begin{pro}[Kemeny \cite{KS76}]\label{pro:Kem}
The values $\l(Q^{\geq 1}(u), u \in V\r)$
are all equal. Their common value, denoted by ${\bf Q}^{\geq 1}$, is finite.

Equivalently, if $X$ is a random variable distributed according to $\rho$ and independent of $C$, the real numbers $\l(\E_u\l(\tau_X^{\geq 1}\r), u \in V\r)$ are all equal (and finite). 
\end{pro}
This constancy has been discussed in many papers, and we refer to some of them here, for historical notes and applications: Kemeny and Snell \cite[Thm. 4.4.10]{KS76}, Levene and Loizou \cite{725624ef-8ff4-3d57-a71b-eb52836bed3b}, Pitman \& Tang \cite{MR3757519}. Further discussions, connections, and references can also be found in Grinstead and Snell \cite[p.469-470]{GS}, Doyle \cite{doyle2009} and Hunter \cite{MR660986}.
 
Notice that the second statement is indeed equivalent to the first because $\E_u\l(\tau_X^{\geq 1}\r)=\sum_{v}  \E_u\l(\tau_v^{\geq 1}\r)\rho_v ={\bf Q}^{\geq 1}\,/\,{\bf Z}$. This means that the expected value of the first positive hitting time of a point chosen according to the stationary distribution does not depend on the starting state of the chain. In the literature the common value of $\E_u\l(\tau_X^{\geq 1}\r) ={\bf Q}^{\geq 1}\,/\,{\bf Z}$ is known as the \textbf{Kemeny constant}.
 
There is an analogue property, concerning the $\tau^{\geq t}_v$, the first visit time of $v$ after time $t$. 
Set 
\[ 
Q^{\geq t}(u) :=\sum_{v}  \E_u\l(\tau_v^{\geq t}\r)\pi_v ={\bf Z}\; \E_u(\tau_X^{\geq t}).
\]
A corollary of \Cref{pro:Kem} is the following:
\begin{cor} For each fixed $t\geq 0$, all the real numbers $(Q^{\geq t}(u) ,u \in V)$ are equal. Moreover, their common value ${\bf Q}^{\geq t}$ satisfies
\ben {\bf Q}^{\geq t}={\bf Z}\,t+{\bf Q}^{\geq 0}.\een
\end{cor}
\begin{proof}For each $t\geq 1$, by the strong Markov property, we have
\be
\E_u\l(\tau_X^{\geq t}\r)&=&t-1+\sum_{v}\P_u(C_{t-1}=v) \E_v\l(\tau_X^{\geq 1}\r)\\
&=&t-1+\sum_{v}\P_u(C_{t-1}=v)\, {\bf Q}^{\geq 1}/{\bf Z}= t-1+{\bf Q}^{\geq 1}/{\bf Z}.
\ee
This expression is independent of $u$.
Hence, ${\bf Q}^{\geq t}={\bf Z}\,(t-1)+{\bf Q}^{\geq 1}$. It remains to prove that ${\bf Q}^{\geq 1} ={\bf Z}\,+ {\bf Q}^{\geq 0}$.
Write
  \be
  \E_u\l(\tau_X^{\geq 0}\r)&=&\E_u\l(\tau_X(\1_{\tau_X=0}+\1_{\tau_X\geq 1})\r)
  =\E_u\l(\tau_X^{\geq 1}\1_{X\neq u}\r)  =\E_u\l(\tau_X^{\geq 1}\r)-\E_u\l(\tau_X^{\geq 1}\1_{X=u}\r)\\
  &=&{\bf Q}^{\geq 1}/{\bf Z}-\E_u\l(\tau_u^{\geq 1}\r)\rho_u={\bf Q}^{\geq 1}/{\bf Z}-1.
  \ee
This last equality comes from the fact that the stationary distribution satisfies $\rho_u=1/\E_u(\tau_u^{\geq 1})$ (see e.g. \cite[Theo.1.7.7]{Norris} and Corollary \ref{cor:der:det}). 
\end{proof} 

\subsection{Main result}

For each $t\geq 0$, consider the probability-generating function of $\tau_{v}^{\geq t}$, starting from $u$:
\ben\label{eq:qsqdd}
\G_{u,v}^{\geq t}(x)&:=& \E_u\l(x^{\tau_{v}^{\geq t}}\r).
\een
It is straightforward to verify that, for all pairs $(u,v)$, $\G_{u,v}^{\geq t}$ seen as the power series $\sum_{k\geq 0} \P(\tau_{v}^{\geq t}=k)x^k$, has a radius of convergence $R_{u,v,t}$ strictly greater \footnote{Irreducibility implies that the graph $(V,E)$ where $E=\{(u,v)~:~M_{u,v}>0\}$ is strongly connected. 
Take $\Delta$ the diameter of the directed graph $(V,E)$, and $p$ the minimum, over all pairs $(u,v)\in V^2$, of the probability to go from $u$ to $v$ in at most $\Delta$ steps. Clearly $p>0$.
Hence, starting from $u$, within each $\Delta$ steps, the probability to hit $v$ is larger than $p$, which implies, {\it in fine}, that ${\cal L}(\tau_v^{\geq 0}~|~C_0=u)$ is smaller for the stochastic order than $\Delta$ multiplied by a geometric random variable $g$ with parameter $p$. The conclusion follows, since the radius of convergence of the probability generating function of   $g$ is  $R = 1/(1-p)$, and then the one of $g \Delta$ is $1/(1-p)^{1/\Delta}$, which is larger than 1.} than 1. Hence, since $V$ is finite, there exists a real number $R_t$ satisfying
\[ 1 < R_t < \min\{ R_{u,v,t}~:~ (u,v)\in V^2\},\] so that all the $\G_{u,v}^{\geq t}$  \underbar{converge normally} in the closed ball  $\bar{B}(0,R_t)$, and even in an open ball containing  $\bar{B}(0,R_t)$. Of course, here and everywhere else, power-series when considered as functions, are defined from (a subset of) $\C$ to $\C$. 
For any $u\in V$, define the polynomial
\ben \pi_u(x)&:=& \det \l(\Id_{d-1} - x\,\M^{(u)}\r).\een
Observe that its evaluation at $x=1$ allows to recover $\pi_u$,
\[\pi_u(1) = \pi_u,~~\forall u\in V.\]
For each $t\geq 0$, each $u\in V$, define the power-series
\ben\label{eq:Kut}
 \bK_u^{\geq t}(x)&=& \sum_{v\in V} \G_{u,v}^{\geq t}(x) \pi_v(x).\label{eq:JF}
 \een
The series $\bK_u^{\geq t}(x)$ (whose coefficients are obtained by taking the Cauchy product of the coefficients of $ \G_{u,v}^{\geq t}(x)$ and those of $\pi_v(x)$) is absolutely convergent on  $\overbar{B}(0,R_t)$ as a product of two absolutely convergent series on this set.
  
The following new theorem provides a \underbar{functional analogue} of the Kemeny constancy theorem: 
\begin{theo}\label{theo:laK}
  \begin{itemize}
  \item[$(i)$] The power series $\l(\bK_u^{\geq 0}(x), u \in V\r)$ are all equal on $\bar{B}(0,R_t)$.
  \item[$(ii)$] Their common values  on $\bar{B}(0,R_t)$ is the polynomial of degree $d-1$,
  \ben\label{eq:laK} \bK^{\geq 0}(x) :={\det\big(\Id_d-x\;\M\big)}\;/\;{(1-x)}.\een 
  \end{itemize}
\end{theo}
\noindent  Since $\M$ has eigenvalue 1,  $\det(\Id_d-M)=0$ and then $(1-x)$ divides the polynomial $\det(\Id_d-x\M)$: it implies that $\bK^{\geq 0}(x)$ is then a polynomial with degree $|V|-1$. 

\begin{rem}\label{rem:warning}   Apart from special cases, $\E_u\l(x^{\tau_v^{\geq 0}}\r)$ is a power series, with an infinite number of terms and radius of convergence $R_{u,v,0}>1$ as already stated. \footnote{The function $G_{u,v}^{\geq t}$ seen as power series, has in general a finite radius of convergence (it is easy to see that all the $G_{u,v}^{\geq 1}$ have an infinite radius of convergence if and only if $\M$ is the transition matrix of a deterministic Markov chain).} Hence the fact that $\bK_u^{\geq 0}(x)$ is a polynomial, reveals huge algebraic cancellations in the sum \eqref{eq:JF}. Since $\bK_{u}^{\geq t}$ is a polynomial and then holomorphic, and the definition  \eref{eq:Kut} is that of a power series, it is tempting to think that the identity given in $(ii)$ is valid over the complete set $\C$, but it is not the case. 
    As a matter of fact, in general, $G_{u,v}^{\geq t}$ does not converge in an interval of the form  $[r_{u,v,t},+\infty)$ for a finite constant $r_{u,v,t}$. Hence in principle $\bK_u^{\geq t}$ is not defined on this set.
    \end{rem}

We have the following Corollary of \Cref{theo:laK}$(i)$:
\begin{cor}\label{cor:gen} For each fixed $t\geq 0$, the power series $\l(\bK_u^{\geq t}(x), u \in V\r)$ are all equal inside $\bar{B}(0,R_t)$, and their common value is 
  \begin{align}\label{eq:imp}
  \bK^{\geq t}(x)=x^t \bK^{\geq 0}(x),
  \end{align}
\end{cor}

 \begin{proof}
Inside $\bar{B}(0,R_t)$ (and even beyond, in an open disk containing this set), by the strong Markov property, $ \E_u\l(x^{\tau_v^{\geq t}}\r)=x^t\sum_{w}\P_u(C_t=w)\E_w\l(x^{\tau_v^{\geq 0}}\r)$, and thus
\ben
\bK^{\geq t}_u(x)=\sum_v\E_u\l(x^{\tau_v^{\geq t}}\r)\pi_v(x)&=&x^t\sum_{w}\P_u(C_t=w) \l(\sum_v\E_w\l(x^{\tau_v^{\geq 0}}\r)\pi_v(x)\r)
= x^t \times 1 \times \bK^{\geq 0}(x).\een
\end{proof}

The proof of \Cref{theo:laK} entirely relies on the following identity, which itself is just an instance of a well known theorem of combinatorics:
\begin{lem}\label{lem:essential} For all $u,v\in V$, for $x$ in the open ball  $B(0,1)$
\ben\label{eq:dqfd} G_{u,v}^{\geq 0}(x)\frac{ \det\big(\Id_{d-1}-x\M^{(v)}\big)}{\det\big(\Id_d-x\M\big)} =\l[\big(\Id_d-x\M\big)^{-1}\r]_{u,v},\een
\end{lem}
\begin{proof}
  By writing  $\l[\big(\Id_d-x\;\M\big)^{-1}\r]_{u,v}=\sum_{k\geq 0} \l[(x\M)^k\r]_{u,v}$ it appears that the right hand side is the generating function of the paths of any length from $u$ to $v$, where the weight of a path $q=(q_0=u,\cdots,q_k=v)$ is  
  \[W(q):=\prod_{j=1}^k (xM_{q_{j-1},q_j}) = x^k\prod_{j=1}^k M_{q_{j-1},q_j},\] or in other words, the product of the elementary weight $w_{a,b}=xM_{a,b}$ associated with each step $(a,b)$. Notice that multiple passages to $v$ or to any vertices are allowed.\par
  Notice that the series $\sum_{k\geq 0} \l[(x\M)^k\r]_{u,v}$ is absolutely convergent in the open ball  $B(0,1)$, but not convergent at 1.  This is a consequence of the ergodic theorem, since $\E_u(\sum_{k=1}^n \1_{C_k=v})=\sum_{k=0}^n \l[\M^k\r]_{u,v}$ behaves at the first order as $n\rho_v$.
  
But such a path $q=(q_0,\cdots,q_k)$ can be decomposed  (univocally) in two paths $(p,s)$ as follows.\\
-- In $p$ we store the prefix of $q$ up to its first passage time $t$ at $v$, and then get $p=(q_0,\cdots,q_t)$.\\
-- In $s$, we store the suffix of the path $q$, from time $t$, that is we set  $s=(q_t,\cdots,q_k)$.

We take the first position $s_0$ of $s$ to be $s_0=q_t=v$ so that the path $s$ is never empty, and it is a cycle from $v$ to $v$ (with multiple passages at any vertices allowed again). Of course, each given step of $p$ is now encoded once either as a step in $q$ or as a step in $s$. Hence
\[W(q) = W(p) \times W(s).\]
Hence,  since $G_{u,v}^{\geq 0}(x)$ is the sum of the weights of all possible prefixes, and $\l[\big(\Id_d-x\M\big)^{-1}\r]_{v,v}$ is the sum of all possible suffixes, we have
\ben\label{eq:decMm1} \l[\big(\Id_d-x\M\big)^{-1}\r]_{u,v}=  G_{u,v}^{\geq 0}(x) \l[\big(\Id_d-x\M\big)^{-1}\r]_{v,v}.\een
Hence it remains to prove that 
\ben  \l[\big(\Id_d-x\M\big)^{-1}\r]_{v,v}= \frac{\det\big(\Id_{d-1}-x\M^{(v)}\big)}{\det\big(\Id_d-x\M\big)},\een
and this is the formula provided by the Laplace method when one computes the inverse of a matrix.
\end{proof}

\subsection{Deduction of  \Cref{theo:laK} from \Cref{lem:essential}}

Let $\Geo(x)$ be a geometric random variable defined on $\{0,1,2,\cdots\}$ whose distribution is 
\[\P(\Geo(x)=k)=x^k(1-x),~~~~~ k\geq 0.\]
This is the common geometric distribution with support $\N:=\{0,1,2,\cdots\}$, and success parameter $1-x$. It is a distribution on $\N$ for $x\in [0,1)$.
Now, let $C_{\Geo(x)}$ be the position of the Markov chain $C$ observed at time $\Geo(x)$, where $\Geo(x)$ is independent of $C$.

Recall that $\pi_v(x)=\det(\Id_{d-1}-x\M^{(v)})$. Multiplying both sides of \eref{eq:dqfd} by $1-x$, we get the equivalent expression  (inside $B(0,1)$)
\be
(1-x)\frac{G_{u,v}^{\geq 0}(x)\pi_v(x)}{\det(\Id_d-x\M)} &=&(1-x)\l[(\Id_d-x\M)^{-1}\r]_{u,v}=\sum_{k\geq 0}(1-x)x^kM^k_{u,v}.\ee
For $x\in[0,1]$, this quantity corresponds to 
\be  \P_u(C_{\Geo(x)}=v).
\ee
Hence, for $x\in[0,1)$, from this identity, Theorem  \ref{theo:laK} is indeed trivial! 
The random variable $C_{\Geo(x)}$ is clearly well defined, and therefore $\sum_{v\in V} \P_u(C_{\Geo(x)}=v)=1$, so that $\dis\sum_{v\in V} (1-x)\frac{G_{u,v}^{\geq 0}(x)\pi_v(x)}{\det(\Id_d-x\M)}=1$, which is independent of $u$, and immediately equivalent to  \Cref{theo:laK}, except for a small detail: we have used \Cref{lem:essential} which is valid inside $B(0,1)$, and the properties of the geometric distribution that are valid on the interval $[0,1)$ only, while proving an identity on the larger set $\bar{B}(0,R_0)$. However, the identity of \Cref{theo:laK} is an identity between holomorphic functions, and we can use the identity theorem to extend this equality from $(0,1)$ to a larger domain, on which we know that the function ${\det\big(\Id_d-x\;\M\big)}\;/\;{(1-x)}$ as well as  $\bK^{\geq 0}$ are analytic.  By the discussion in \Cref{rem:warning} (the fact that $G_{u,v}^{\geq t}$ is not well defined on $\C$, but is well defined and analytic on $\overbar{B}(0,R_t)$), we deduce that the two functions coincide on $\bar{B}(0,R_0)$ so that \Cref{theo:laK} holds.
 
\subsection{Deduction of Kemeny's constancy Theorem from Theorem  \ref{theo:laK}}

For a power series ${\bf H}(x)=\sum_{m\geq 0} H_mx^m$, the $m$th coefficient $[x^m]{\bf H}(x)$ can be extracted using the $m$th derivative. More precisely, $H_m ={\bf H}^{(m)}(0)/m!$. Hence 
 \ben\label{eq:ssum}  \pi^{(m)}_v(0)/m!= [x^m]\det\big(\Id_{d-1}-x\M^{(v)}\big),\een
and we have also $\pi_v^{(0)}(0)=\pi_v(0)=1$ for all $v\in V$, and all the $\pi^{(m)}_v(0)/m!$ are equal to zero for $m\geq d$, since $\M^{(v)}$ has size $d-1$.

Theorem  \ref{theo:laK} states an equality of generating functions on non trivial balls, and its independence of $u$. This equality  implies equality of their $m$th coefficients. 
 
Denote by $K_{u,m}^{\geq t}$ the $m$th coefficient of $\bK_{u}^{\geq t}$. Since by Theorem  \ref{theo:laK}, $K_{u,m}^{\geq t}$ does not depend on $u$, denote by $K_{m}^{\geq t}$ their common value. The  number $K_m^{\geq t}$  can be expressed as a \underbar{finite sum}:
\ben\label{eq:qfgrfyk}K_{u,m}^{\geq t}=K_m^{\geq t} = \sum_{s=0}^{m \wedge (|V|-1)} \sum_{v\in V}\P_u\l(\tau_v^{\geq t}=m-s\r) \frac{\pi^{(s)}_v(0)}{s!}.\een

Take $t=1$, and write
\ben\label{eq:geid}  \sum_m m K_{u,m}^{\geq 1} &=& \sum_m m K_{m}^{\geq 1} \\&=& \sum_{r\geq 0} \sum_{s=0}^{|V|-1} \sum_{v\in V}\P_u\l(\tau_v^{\geq 1}=r\r) (r+s)\frac{\pi^{(s)}_v(0)}{s!}\\
\label{eq:geid2}&=&  \sum_{v\in V}\l(\sum_{r\geq 0} \P_u\l(\tau_v^{\geq 1}=r\r) r\r)\sum_{s=0}^{|V|-1}\frac{\pi^{(s)}_v(0)}{s!} +  \sum_{v\in V}\l(\sum_{r\geq 0}\P_u\l(\tau_v^{\geq 1}=r\r)\r)  \sum_{s=0}^{|V|-1} s\frac{\pi^{(s)}_v(0)}{s!}\\
\label{eq:geid3}&=& \E_u\bigl(\tau_X^{\geq 1}\bigr)   +
  \sum_{v\in V} \sum_{s=1}^{|V|-1}\frac{\pi^{(s)}_v(0)}{(s-1)!}.
\een
In this last identity, only $\E_u\bigl(\tau_X^{\geq 1}\bigr)$ is a function of $u$, and therefore it must be constant.

\section{Discussion and consequences}

\subsection{The sequence $(\P_u(\tau_v^{\geq 0}=m),m\geq 0)$ satisfies a finite linear recursion}

In the proof of \Cref{lem:essential}, we have established that: 
\begin{cor}\label{cor:basic}
    The matrix $\bma G_{u,v}^{\geq 0}(x)\pi_v(x)\ema_{u,v\in |V|}$ is equal to the adjugate matrix $\Adj(\Id_d-x\M)$ for $x\in B(0,1)$.
\end{cor}
Notice that the coefficients of  $\Adj(\Id_d-x\M)$ are polynomials of degree $|V|-1$. Developing $G_{u,v}^{\geq 0}(x).\pi_v(x)$ one gets that 
\[
    [x^m]G_{u,v}^{\geq 0}(x).\pi_v(x) = \sum_{s=0}^{m\wedge (|V|-1)} \P_u(\tau_v^{\geq 0} = m-s)\frac{\pi_v^{(s)}(0)}{s!}.
\]
Therefore if $m\geq |V|$, the coefficient $[x^m]G_{u,v}^{\geq 0}(x)\pi_v(x) = \sum_{s=0}^{m\wedge (|V|-1)} \P_u(\tau_v^{\geq 0} = m-s)\frac{\pi_v^{(s)}(0)}{s!}=0$. Using the fact that the degree of $\pi_v(x)$ is $|V|-1$, we obtain the following recursion defining the law of $\tau_v^{\geq 0}$:
\begin{lem} For $m>|V|-1$,

  \[  \P_u(\tau_v^{\geq 0} = m)= -\sum_{s=1}^{|V|-1} \P_u(\tau_v^{\geq 0} = m-s)\frac{\pi_v^{(s)}(0)}{s!}\]
\end{lem}\begin{proof}
The only argument of the proof is that $
    \sum_{s=0}^{|V|-1} \P_u(\tau_v^{\geq 0} = m-s)\frac{\pi_v^{(s)}(0)}{s!}=0$
    and $\pi_v^{(0)}(0)/0!=1$. 
  \end{proof}

\subsection{Factorial moments of hitting times identities}
    
Consider the descending factorial numbers defined by 
\ben{\sf Fac}_k(n)=n!/(n-k)!=n(n-1)\cdots (n-k+1), \textrm{ for } n\geq 1  \textrm{ and }k\geq 0 \een
and ${\sf Fac}_0(n)=1$. When we have the probability generating function $g(x)=\E(x^Y)$ of a random variable taking its values in $\N=\{0,1,2,\cdots\}$, its derivative gives access to the factorial moments, by the formula $g^{(m)}(1)= \E({\sf Fac}_m(Y))$. Hence, for $d\geq 0$ and any $u,v\in V$,
\ben \label{eq:fac} 
\G_{u,v}^{\geq 1, (d)}(1)  
&=& \E_u\l(\l(\tau^{\geq 1}_v-0\r)\cdots\l(\tau^{\geq 1}_{v}-d+1\r)\r)=\E_u\l({\sf Fac}_d(\tau^{\geq 1}_v)\r)<+\infty;
\een
notice that this formula is indeed valid also for $d=0$ because $\G^{(0)}_{u,v}(1)=\E_u\l(x^{\tau_v^{\geq 1}}\r)~|_{x=1}=1=\E_u\l({\sf Fac}_0(\tau^{\geq 1}_v)\r)$ since $\tau^{\geq 1}_v$ is almost surely finite. 
By differentiating \eqref{eq:JF} $n$ times and then evaluating at $x=1$, we obtain
\ben  \bK^{\geq 1,(n)}(1)= \sum_{s=0}^n \binom{n}{s} \G_{u,v}^{\geq 1,(s)}(1)\, \pi_v^{(n-s)}(1).\een

\begin{cor}\label{cor:1} For all $n\geq 0$, all $u\in V$,
  \ben\label{eq:qgt}
  \bK^{\geq 1,(n)}(1)/n! 
  &=& \sum_{s=0}^{n\wedge (|V|-1)}  \sum_{v\in V} \frac{\E_u\l({\sf Fac}_{n-s}(\tau^{\geq 1}_v)\r)}{(n-s)!}  \;  \frac{\pi^{(s)}_v(1)}{s!}.
\een

In particular,  for $n=1$, for all $u\in V$,
\ben
\bK^{\geq 1,(1)}(1)   &=& \sum_{v\in V} \E_u(\tau_v^{\geq 1})\pi_v + \sum_{v\in V}\pi^{(1)}_v(1), 
\een
so that ${\bf Q}^{\geq 1}(u)=\sum_{v\in V} \E_u(\tau_v^{\geq 1})\pi_v$ does not depend on $u$ (this is the Kemeny theorem).
\end{cor}
Since $\bK^{\geq 1}(x)$ is a polynomial of degree $|V|$, for $n\geq |V|+1$
  the left hand side of \eref{eq:qgt} becomes zero, so that the ``$s=0$ term'' in the right hand side is the opposite of the sum of the others. For  $n\geq |V|+1$,
  \[ \sum_{v\in V} \frac{\E_u\l({\sf Fac}_{n}(\tau^{\geq 1}_v)\r)}{n!}  \;  \pi^{(0)}_v(1) =\bZ\, \E_u\l(\frac{{\sf Fac}_{n}(\tau^{\geq 1}_X)}{n!}\r)=-
    \sum_{s=1}^{ |V|-1}  \sum_{v\in V} \frac{\E_u\l({\sf Fac}_{n-s}(\tau^{\geq 1}_v)\r)}{(n-s)!}  \;  \frac{\pi^{(s)}_v(1)}{s!},\]
  and again, this does not depend on $u$.

  \begin{rem}[About the variance of $\tau_X^{\geq 1}$]
 Take $n=2$ in \Cref{cor:1}. For $|V|\geq 3$, we get that 
      \ben\label{eq:efdqs}
       \bK^{\geq 1,(2)}(1)/2!&=&\frac{\bZ}2\,\E_u\l(\tau_X^{\geq 1}\l(\tau_X^{\geq 1}-1\r)\r)+\sum_{v\in V}\E_u\l(\tau_v^{\geq 1} \r)\pi_v^{(1)}(1)+ \sum_{v\in V}\frac{\pi_v^{(2)}(1)}{2}
       \een
       is independent of $u$. Since $\E_u(\tau^{\geq 1}_X)$, is independent of $u$ too, we preserve the independence with respect to $u$ of the right hand side of \eref{eq:efdqs} by adding $\frac{\bZ}2\l(\E_u(\tau^{\geq 1}_X)-\E_u(\tau^{\geq 1}_X)^2\r)-\sum_{v\in V} \frac{\pi_v^{(2)}(1)}{2}$. Therefore, there exists a constant $\mathfrak{C}$ (depending on $\M$ only) such that 
\[\mathfrak{C}:=\frac{\bZ}2\,{\sf Var}_u(\tau^{\geq 1}_X)+ \sum_{v\in V}\E_u\l(\tau_v^{\geq 1} \r)\pi_v^{(1)}(1) \]       
In general (and this can be checked on examples), $\sum_{v\in V}\E_u\l(\tau_v^{\geq 1} \r)\pi_v^{(1)}(1)$ depends on $u$, so that, ${\sf Var}_u(\tau^{\geq 1}_X)$ depends on $u$ too.
Nevertheless, we have an identity
\[\max_{u\in V} {\sf Var}_u(\tau^{\geq 1}_X) - \min_{u \in V}  {\sf Var}_{u}(\tau^{\geq 1}_X) =\frac{2}{\bZ}\l[ \max_u \sum_{v\in V}\E_u\l(\tau_v^{\geq 1} \r)\pi_v^{(1)}(1)-\min_{u\in V} \sum_{v\in V}\E_u\l(\tau_v^{\geq 1} \r)\pi_v^{(1)}(1)\r].\]
\end{rem}

\subsection{Alternative algebraic representation of $\G_{u,v}^{\geq t}$ and consequences}

The generating functions $\G_{u,v}^{\geq t}(x)$ can be expressed using simple matrix operations (for $x\in \overbar{B}(0,R_t)$): letting  $\M^{\star v}$ be the matrix that coincides with $\M$ except for its $v$-th column which contains only zero values. For $t\geq 1$, we get
\[\G_{u,v}^{\geq t}(x)= \l[(x\M)^{t-1}\l(\sum_{k\geq 0} (x\M^{\star v})^{k}\r)(x\M)\r]_{u,v}=   \l[(x\M)^{t-1}(\Id_d -x\M^{\star v})^{-1}(x\M)\r]_{u,v},\]
since a path $p=(p_0,\dots,p_k)$ from $u$ to $v$ contributes to $\G_{u,v}^{\geq t}(x)$ if :\\
		$\bullet$ $k\geq t$,\\
		$\bullet$ it starts at $u$, meaning that $p_0=u$,\\
		$\bullet$ it avoids $v$ from time $t$ to $k-1$, i.e. $p_t,\cdots,p_{k-1}$ are different from $v$ and\\
		$\bullet$ it ends at $v$, this is $p_k=v$.\\
For $t=0$, there is a difference since
\ben\label{eq:Giojj} \G_{u,v}^{\geq 0}(x)=  \1_{u=v}+\1_{u\neq v} \l[(\Id_d-x\M^{\star v})^{-1}(x\M)\r]_{u,v}\een
and these identities are valid at least for $x\in (-1,1)$.
Taking into account that for $v=u$, 
$\l[(\Id_d-x\M^{\star v})^{-1}(x\M)\r]_{u,v}=\l[(\Id_d-x\M^{\star v})^{-1}(x\M)\r]_{v,v}$, which represents the generating function of paths of size at least one with first hitting time of $v$ greater than 1, that is, it is $\E_v(x^{\tau^{\geq 1}_v})$. 
we get
\ben\label{eq:Giojjj} \G_{u,v}^{\geq 0}(x)&=&  \1_{u=v}\l(1-\E_v\l(x^{\tau^{\geq 1}_v}\r)\r) + \l[(\Id_d-x\M^{\star v})^{-1}(x\M)\r]_{u,v}\\
&=&  \1_{u=v}\l(1-\E_v\l(x^{\tau^{\geq 1}_v}\r)\r) + \G_{u,v}^{\geq 1}(x).
\een
Multiplying by $\pi_v(x)$ and summing over $v$ for fixed $u$ gives,
\[\bK^{\geq 0}_u(x)= \pi_u(x)\l(1-\E_u\l(x^{\tau^{\geq 1}_u}\r)\r) +   \bK^{\geq 1}_u(x).\]
Now, we can divide both size by $\bK^{\geq 0}_u(x)$ (and recall that $\bK^{\geq 1}_u(x)=x\bK^{\geq 0}_u(x)$) and get the relation
\begin{align}\label{invx}
1-\E_u\l(x^{\tau^{\geq 1}_u}\r)=(1-x)\frac{\bK^{\geq 0}(x)}{\pi_u(x)} = \frac{\det(\Id_d -x\M)}{\pi_u(x)}=\frac{1}{(\Id_d-x\M)^{-1}_{u,u}}
\end{align}
so that
\begin{cor}\label{cor:fqqqf} For all $u \in V$ and $x\in B(0,R_{1})$
  \[\E_u\l(x^{\tau^{\geq 1}_u}\r)=1-\frac{1}{(\Id_d-x\M)^{-1}_{u,u}},\]
  moreover, for all $u,v \in V$
  \[\l[(\Id_d - x\M)^{-1}\r]_{u,v}= {\E_u\l(x^{\tau^{\geq 0}_v}\r)}/\l(1-\E_v\l(x^{\tau^{\geq 1}_v}\r)\r).\]
  \end{cor}
  \begin{proof}  The first statement is a consequence of the previous discussion, while for the second, we start from 
    \eqref{eq:decMm1}, use $G_{u,v}^{\geq 0}(x)=\E_u\l(x^{\tau^{\geq 0}_v}\r)$ and from the first statement of the present corollary.
  \end{proof} 
  Observe that since $\M$ is an irreducible transition matrix, for $x\in(0,1)$,
  \begin{align}\label{MTx}
  	N_x&:=\sum_{k\geq 0} (1-x) x^k M^k=(1-x)(\Id_d-xM)^{-1}
  \end{align} 
  is also an irreducible transition matrix, and then the sum of the coefficient on each row is one. Hence, \Cref{cor:fqqqf} yields to  
  \begin{cor}\label{cor:constancy} For all $u\in V$ and $x\in B(0,1)$,
    $$\sum_{v\in V} {\E_u\l(x^{\tau^{\geq 0}_v}\r)}/\l({1-\E_v\l(x^{\tau^{\geq 1}_v}\r)}\r)={1}/({1-x}).$$ 
  \end{cor}
  \begin{rem}
  	Corollary \ref{cor:constancy}, states also a constancy result, since the right hand side is independent of $u\in V$.
  \end{rem}
  \begin{rem}
  	The matrix defined on equation \eqref{MTx} is of independent interest :  in the irreducible case when taking the limit at $x=1$ it gives the invariant measure matrix whose rows are equal to the invariant measure. This matrix is intimately related with the group inverse (see \cite{M75}) and the fundamental matrix (see \cite{KS76}); both matrices give access to the matrix whose coefficients are the expectations of first passage times indexed by initial and target points, among others (see \cite{M75,AFH02}). 
  \end{rem}

 \begin{cor}\label{cor:der:det}
	\[
		\frac{d}{dx}\det(\Id_d-x\M) = -\frac{1}{x}\sum_{v\in V}\E_v\l(x^{\tau_v^{\geq 1}}\r)\pi_v(x)
	\]
\end{cor}

\begin{proof}
	By the Jacobi formula for the derivative of a determinant one has
	\be 
	\frac{d}{dx}\det(\Id_d-x\M) 
	&=& \Tr(\Adj(\Id_d-x\M)(-\M))\\
	&=& \frac{1}{x}\Tr\big(\Adj(\Id_d-x\M)(\Id_d-x\M) - \Adj(\Id_d-x\M)\big)\\
	&=& \frac{1}{x}\Tr\big(\det(\Id_d-x\M)\Id - \Adj(\Id_d-x\M)\big)\\
	&=& \frac{1}{x}\sum_{v\in V}\left(\det(\Id_d-x\M) - \pi_v(x)\right)\\
	&=& -\frac{1}{x}\sum_{v\in V}\E_v\l(x^{\tau_v^{\geq 1}}\r)\pi_v(x),
 	\ee
	where in the last equality we use equation \eqref{invx} to obtain that $\det(\Id_d-x\M) = \l(1-\E_v\l(x^{\tau_v^{\geq 1}}\r)\r)\pi_v(x)$.
\end{proof}

We may deduce from here a very famous result: 
\begin{cor}\label{cor:hsg} For all $v\in V$,
  \[		\rho_v=\frac{\pi_v}{\sum_u\pi_u} = \frac{1}{\E_v(\tau_v^{\geq 1})}.	\]
\end{cor}
 
\begin{proof} 
  First, we use \Cref{cor:der:det} and evaluate at $x=1$. We get
  \ben\label{eq:ytutp1}	 
  \frac{d}{d x}\det(\Id_d-x\M)~\Big|_{x=1} =  -\frac{1}{x} \sum_{v\in V}\E_v\left(x ^{\tau_v^{\geq 1}}\right)\pi_v(x)~\Big|_{x=1}= -\sum_{v\in V} \pi_v
  \een
  Second, we take the derivative of $\det(\Id_d-x\M) = \l(1-\E_v\l(x^{\tau_v^{\geq 1}}\r)\r)\pi_v(x)$ (the second equality in equation \eqref{invx}) and evaluate at $x=1$ to obtain 
	\begin{align}\label{eq:ytutp2} 
      \frac{d}{d x}\det(\Id_d-x\M)~\Big|_{x=1}=\l(-\E_u(\tau_u^{\geq 1} x^{\tau_u^{\geq 1} -1})\pi_u(x) + (1-\E_u(x^{\tau_u^{\geq 1}}))\pi_u'(x)\r)\Big|_{x=1}= -\E_u(\tau_u^{\geq 1})\pi_u
	\end{align}
	which implies the result.
\end{proof}

  \subsection{On the nature of $\bK^{\geq t}_u(x)$}

In general, the sign of $\pi_v(x)$, as a polynomial in a real variable $x$, depend of $x$ and of $v$. In general, the signs of $\pi_v(x)$ and of $\pi_{v'}(x)$, for $v\neq v'$, are not the same for all $x\in \R$ (see an example in \Cref{sec:ae}). However, for $x\in(0,1)$, the $\pi_v(x)$ are non negative, and then $\bK^{\geq 0}(x)$ has a probabilistic interpretation:
\begin{pro}\label{pro:qdt}
  \bir
  \itr For every $x\in (0,1)$ 
  the vector  $(\pi_v(x), v\in V)$ has positive coordinates, so that it is proportional to a probability measure $\mu_x$ on $V$.
  \itr Hence, for $X_x\sim \mu_x$ independent of the Markov chain $C$, the values of $\l(\E_u\l(a^{\tau_{X_a}^{\geq t}}\r),u\in V\r)$ are all the same.
  \eir
\end{pro}

The second assertion is a consequence of \Cref{cor:gen}.

\begin{proof} Proof of $(i)$.
  Fix $v\in V$, and take $x\in(0,1)$. We will show that $\pi_v(x)>0$.  
  Consider the matrix $\widetilde{\M}(v)$ defined by
  \[\widetilde{\M}(v,x) = x \M + (1-x) C(v)\]
  where $C(v)$ is the matrix indexed by $V$, with zero coefficients except on the column corresponding to the state $v$, which contains ones.
  \color{black}
  It is easy to see that $\widetilde{\M}$ is a stochastic irreducible matrix,
  and then
  \[\tilde{\pi}(v):=\bma \det(\Id_{d-1}-\widetilde{\M}(v,x)^{(u)}) , u \in V \ema\] is proportional to the invariant measure of $\widetilde{\M}(v,x)$, and then all its coordinates are positive. 
  The coefficient $\tilde{\pi}_v(v)$ corresponding to the vertex $v$, is
  \be \tilde{\pi}_v(v)&=&\det(\Id_{d-1}-\widetilde{\M}{(v,x)}^{(v)})\\
  &=&\det(\Id_{d-1}-x\M^{(v)})=\pi_v(x),\ee therefore $\pi_v(x)$ is positive.\\
  We have proved that for all $v\in V$ the values of $\pi_v(x)$ are positive and therefore there exists a probability distribution $\mu_x$ proportional to $(\pi_v(x),v\in V)$.\\
  Proof of $(ii)$:
Since \[\bK^{\geq t}(x)=\sum_v \E_u\l(x^{\tau_v^{\geq t}}\r)\det(\Id_{d-1}-x\M^{(v)})= \E_u(x^{\tau_{X_x}^{\geq t}}) Z(x),\] where $Z(x)=\sum_{u\in V}\det(\Id_{d-1}-x\M^{(u)})$, the conclusion follows.
\end{proof} 

\subsection{Eigenvalues considerations}\label{eigen}

We will extract some information from the characteristic polynomial formula
$\det(t \Id-\M)=\prod_{i=1}^{|V|} (t - \lambda_i)$ 
where $\lambda_0=1, \lambda_1,\cdots,\lambda_{d-1}$  are the eigenvalues of $\M$.  Hence, for all $x\neq 0$, by taking $t=1/x$
\be
 \det(\Id_d - x\M)=x^d \det(\Id_d/x-\M)=\prod_{i=0}^{d-1} (1 - x\lambda_i).
\ee
Since $\M$ is irreducible, it admits 1 as eigenvalue with multiplicity one, and then
\[\bK^{\geq 0}(x)=\prod_{i=1}^{d-1} (1 - x\lambda_i).\] 
We immediately get 
\begin{align}\label{normalisation}
\prod_{i=1}^{d-1} (1 - \lambda_i)  &= \bK^{\geq 0}(1) = \sum_{v\in V}\pi_v = \bZ.
\end{align}
The following proposition is well known, but the standard proof uses the group inverse \cite{CKN10}. We propose here a proof that avoids this object.

\begin{pro}\label{prop:eigKem}
	The Kemeny constant is  
	\[
		\frac{\bQ^{\geq 1}}{\bZ} = 1 + \sum_{i=1}^{d-1}\frac{1}{1-\lambda_i}
	\]
\end{pro}
We start with a lemma
\begin{lem}\label{pro:secder}
	\[
	\frac{d^2}{dx^2}\det(\Id-x\M)\Big|_{x=1} = (1-d)\bZ-\sum_{v\in V}\pi_v^{(1)}(1)
	\]
\end{lem}
\begin{proof}
	From \Cref{cor:der:det} one gets that
	\be
	\frac{d^2}{dx^2}\det(\Id-x\M)\Big|_{x=1} 
	&=& \left(\frac{1}{x^2}\sum_{v\in V}\E_v\l(x^{\tau_v^{\geq 1}}\r)\pi_v(x) -\frac{1}{x}\sum_{v\in V}\E_v\l(\tau_v^{\geq 1} x^{\tau_v^{\geq 1}-1}\r)\pi_v(x) -\frac{1}{x}\sum_{v\in V}\E_v\l(x^{\tau_v^{\geq 1}}\r)\pi_v^{(1)}(x)\right)\Big|_{x=1}\\
	&=&\sum_{v\in V} \pi_v-\sum_{v\in V} \E_v\l(\tau_v^{\geq 1}\r)\pi_v - \sum_{v\in V}\pi_v^{(1)}(1)\\
	&=&(1-d)\bZ- \sum_{v\in V}\pi_v^{(1)}(1),
	\ee
	where in the last equality we applied Corollary \ref{cor:hsg}.
\end{proof}

\begin{proof}[Proof of Proposition \ref{prop:eigKem}]
	Consider equation \eqref{eq:efdqs}, which says that
	\[
	\bQ^{\geq 1}  = \bK^{\geq 1,(1)}(1) - \sum_{v\in V}\pi^{(1)}_v(1).
	\]
	We have from Theorem \ref{theo:laK} and Equation \ref{eq:imp} that $\bK^{\geq 1}(x) = x\bK^{\geq 0}(x)$.\\
	Setting $A(x) := \left(-\sum_{i=1}^{d-1}\frac{\lambda_i}{1-\lambda_ix}\right)$ it is easy to see verify that: 
	\begin{align}
		\bK^{\geq 1,(1)}(x) &= \left(1+xA(x)\right)\bK^{\geq 0}(x)\\
		\frac{d}{dx}\det(\Id-x\M) &= \frac{d}{dx}\left((1-x)\bK^{\geq 0}(x)\right) = ((1-x)A(x)-1)\bK^{\geq 0}(x)\\
		\frac{d^2}{dx^2}\det(\Id-x\M) & =  \left((1-x)\left(\frac{d}{dx}A(x)+A(x)^2\right)-2A(x)\right)\bK^{\geq 0}(x).
	\end{align}
	Gathering this, evaluating at $x=1$ and using Lemma \ref{pro:secder} we obtain
	\begin{align}
		\bQ^{\geq 1} 
		&= \left(1-A(1)\right)\bK^{\geq 0}(1) + (d-1)\bZ\\
		&= \left(1+\sum_{i=1}^{d-1}\frac{\lambda_i}{1-\lambda_i}\right)\bK^{\geq 0}(1)+(d-1)\bZ\\
		&= \left(2-d+\sum_{i=1}^{d-1}\frac{1}{1-\lambda_i}\right)\bZ+(d-1)\bZ,
	\end{align}
	where in the last equality we used equation \eqref{normalisation}. Dividing by $\bZ$, we conclude.
\end{proof}

\section{An example}
\label{sec:ae}
Consider the irreducible transition matrix
\[\M= \frac{1}{12}\left[ \begin {array}{cccc} 5&2&4&1\\ \noalign{\medskip}1&3&3&5
\\ \noalign{\medskip}1&6&4&1\\ \noalign{\medskip}2&1&6&3\end {array}
\right]\]
whose unique invariant probability distribution is
\[\rho = \frac{1}{1376} \left[ \begin {array}{cccc} 209&396&475&296\end {array} \right]. \]
The computation of $\pi_v(x)$ for $v$ from 1 to 4 gives the folowing result 
\[\left[ \begin {array}{cccc} 1-{\frac {5\,x}{6}}+{\frac {{x}^{2}}{36}}
-{\frac {127\,{x}^{3}}{1728}}&1-x+{\frac {35\,{x}^{2}}{144}}-{\frac {{
x}^{3}}{72}}&1-{\frac {11\,x}{12}}+{\frac {5\,{x}^{2}}{24}}-{\frac {29
\,{x}^{3}}{1728}}&1-x+{\frac {23\,{x}^{2}}{144}}+{\frac {5\,{x}^{3}}{
                   432}}\end {array} \right] \]
         while the matrix $G^{\geq 0}(x)$ computed using \eref{eq:Giojj} gives
 \[G^{\geq 0}(x)= \left[ \begin {array}{cccc} 1&{\frac {x \left( 2\,{x}^{2}-11\,x-24
 \right) }{ \left( x-12 \right)  \left( -4+x \right)  \left( 2\,x-3
 \right) }}&{\frac {-x \left( 43\,{x}^{2}-144\,x+576 \right) }{29\,{x}
^{3}-360\,{x}^{2}+1584\,x-1728}}&{\frac {x \left( 17\,{x}^{2}+21\,x+36
 \right) }{5\,{x}^{3}+69\,{x}^{2}-432\,x+432}}\\ \noalign{\medskip}{
\frac {x \left( 7\,x+12 \right)  \left( x-12 \right) }{127\,{x}^{3}-48
\,{x}^{2}+1440\,x-1728}}&1&{\frac { \left( 7\,x+12 \right) x \left( -
36+11\,x \right) }{29\,{x}^{3}-360\,{x}^{2}+1584\,x-1728}}&{\frac {
 \left( 17\,{x}^{2}-123\,x+180 \right) x}{5\,{x}^{3}+69\,{x}^{2}-432\,
x+432}}\\ \noalign{\medskip}{\frac {-x \left( 41\,{x}^{2}+24\,x+144
 \right) }{127\,{x}^{3}-48\,{x}^{2}+1440\,x-1728}}&\,{\frac {-3
 \left( 2\,{x}^{2}-15\,x+24 \right) x}{ \left( x-12 \right)  \left( -4
+x \right)  \left( 2\,x-3 \right) }}&1&{\frac { \left( -31\,{x}^{2}-69
\,x-36 \right) x}{5\,{x}^{3}+69\,{x}^{2}-432\,x+432}}
\\ \noalign{\medskip}{\frac {-x \left( -24+5\,x \right)  \left( x-12
 \right) }{127\,{x}^{3}-48\,{x}^{2}+1440\,x-1728}}&{\frac { \left( 10
\,{x}^{2}-31\,x-12 \right) x}{ \left( x-12 \right)  \left( -4+x
 \right)  \left( 2\,x-3 \right) }}&-{\frac {x \left( -36+11\,x
 \right)  \left( -24+5\,x \right) }{29\,{x}^{3}-360\,{x}^{2}+1584\,x-
1728}}&1\end {array} \right]\]
and
\[\det(\Id_d-x \M) = {\frac { \left( -1+x \right)  \left( 5\,{x}^{3}-15\,{x}^{2}+54\,x-216
 \right) }{216}}.\]
Finally, for each $u$, the values of $\sum_v G_{u,v}^{\geq 0}(x) \pi_v(x)$ can be checked to be equal to $\det(\Id-x \M)/(1-x)$, and the statement of \Cref{lem:essential} can also be ``checked'' on this example.

\begin{figure}[h!]
  \centerline{\includegraphics[width=10cm]{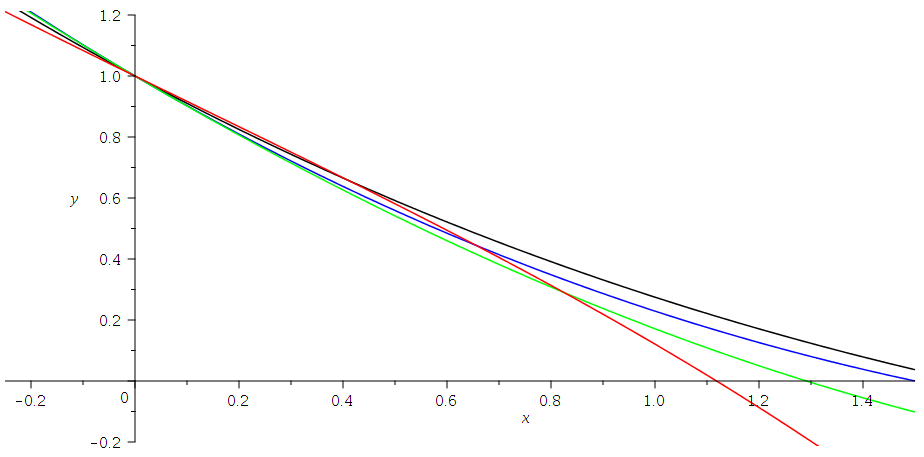}}
  \captionn{The four function $\pi_j(x)$, for $j$ from 1 to 4, with colors red, blue, black and green in this order. They are positive for $x\in[0,1]$ (as proved in \Cref{pro:qdt}). }
\end{figure}

\bibliographystyle{abbrv}

\end{document}